\documentclass[11pt]{article}
\usepackage{amsmath}
\usepackage{amsthm,amsfonts,amssymb,latexsym}
\usepackage[utf8]{inputenc}
\usepackage[english]{babel}
\usepackage{color}

\newtheorem{theorem}{Theorem}
\newtheorem{proposition}[theorem]{Proposition}

\theoremstyle{definition}
\newtheorem{defn}{Definition}
\theoremstyle{remark}

\newtheorem{remark}{Remark}[theorem]
\def\R{{\mathbb R}}
\def\N{{\mathbb N}}
\newcommand{\xR}{{]}{-\infty},+\infty]}
\newcommand{\Rex}{\xR}
\newcommand{\Rb}{\overline{\R}}
\newcommand{\Fcal}{\mathcal{F}}

\newcommand{\ps}{\smallbreak}

\newcommand{\PT}{$\bullet$\ }

\newcommand{\lsc}{lsc}

\newcommand{\del}{\partial}

\newcommand{\delc}{\widehat{\del}}

\newcommand{\dom} {{\rm dom} \kern.15em}
\newcommand{\tq}{:}
\newcommand{\la}{\langle}
\newcommand{\ra}{\rangle}

\newcommand{\eps}{\varepsilon}
\newcommand{\bx}{\bar{x}}
\newcommand{\xb}{\bar{x}}
\newcommand{\yb}{\bar{y}}

\newcommand{\tow}{{\stackrel{w^*}{\longrightarrow}}\;}

\begin{document}
\thispagestyle{empty}
\begin{center}
{\large\bf\sc Upper semismooth functions and the subdifferential determination
property}
\medskip\\
\today
\end{center}

\begin{center}
  {\small\begin{tabular}{c}
  Marc Lassonde\\
   Universit\'e des Antilles, BP 150, 97159 Pointe \`a Pitre, France; and\\
   LIMOS, Universit\'e Blaise Pascal, 63000 Clermont-Ferrand, France\\
  E-mail: marc.lassonde@univ-ag.fr
  \end{tabular}}
\end{center}

\begin{center}
  {\small\it Dedicated to the memory of Jon Borwein.}
\end{center}

\medbreak\noindent
\textbf{Abstract.}
In this paper, an upper semismooth function is defined to be
a lower semicontinuous function
whose radial subderivative satisfies a mild directional upper
semicontinuity property.
Examples of upper semismooth functions are the proper lower semicontinuous
convex functions,
the lower-C$^1$ functions, the regular directionally Lipschitz functions,
the Mifflin semismooth functions,
the Thibault-Zagrodny directionally stable functions.
It is shown that the radial subderivative of such functions
can be recovered from any subdifferential of the function.
It is also shown that these functions are subdifferentially determined,
in the sense that if two functions have the same subdifferential
and one of the functions is upper semismooth, then the two functions
are equal up to an additive constant.
\medbreak\noindent
  
\medbreak\noindent
\section{Introduction}\label{intro}
Jon Borwein discussing generalisations in the area of nonsmooth optimisation
\cite[p.\ 4]{Bor16} writes:
{\small
\begin{quote}
In his thesis Francis Clarke extended Moreau-Rockafellar max formula
to all locally Lipschitz functions. Clarke replaced
the right Dini directional derivative $D_hf(x)$ by
$$
D^c_hf(x) = \limsup_{0<t\to 0,y\to x}
\frac{f(y+th)-f(y)}{t}.
$$
Somewhat miraculously the mapping $p$ sending $h \to D^c_hf(x)$
is always continuous and sublinear in $h$, and so if we define
$\del^Cf(x)=\del p(0)=\{y\in X^*\tq \la y, h\ra\le D^c_hf(x), \forall h\in X\}$, Moreau-Rockafellar max formula leads directly to:
\medbreak\noindent
\textbf{Theorem} (Clarke).
{\it Let $f: E\to\R$ be a locally Lipschitz function. Then
\begin{equation}\label{intro1}
D^c_hf(x)=\sup_{y\in \del^Cf(x)}\la y, h\ra
\end{equation}
for all $h\in E$. [...]
\if{In particular $\del^Cf(x)$ is nonempty. Moreover, $\del^Cf(x)$
is a singleton if and only if $f$ is strictly differentiable at $x$.}\fi}
\medbreak\noindent
In truth Clarke, appealing to Rademacher’s theorem, originally defined $\del^Cf(x)$
as the closed convex hull of limits of nearby points of differentiability.
This makes (\ref{intro1}) seem even more remarkable.
There is, however, a dark side to the situation \cite[Cor.\ 9]{BMW01}.
Recall that a set in a Banach space is generic if it contains intersection of countably many dense open sets. The complement is thus very
small topologically.

\medbreak\noindent
\textbf{Theorem} (Generic triviality \cite{BMW01}).
{\it Let $A$ be an open subset
of a Banach space $X$. Then the
set of non-expansive functions on $A$ with $\del^Cf(x)\equiv B_{X^*}$
for all $x$ in $A$ is generic in the uniform norm on $A$.}
\medbreak\noindent
In other words, in the sense of Baire category the Clarke subdifferential
(likewise the limiting subdifferential in the separable case) of almost
all functions contains no information at any point
of $A$. [...]
\if{Of course none of these functions are convex since then the function is generically strictly Fr\'echet differentiable.
This is the cost of abstraction --- if a construction always works for a
very broad class it usually works only passingly well.}\fi
So, for most Lipschitz functions the Clarke calculus is vacuous.
That is why serious researchers
work with well structured subclasses such \textit{semi-algebraic, partially smooth}
or \textit{essentially smooth} functions.
\end{quote}
}

The fact that subdifferentials cannot discriminate between functions
is a serious drawback according to
Terry Rockafellar \cite{Roc82} who argues:

{\small
\begin{quote}
In subgradient optimization, interest centers on methods for
minimizing $f$ that are based on being able to generate for each $x$
at least one (but not necessarily every) $y\in \del^Cf(x)$, or perhaps just
an approximation of such a vector $y$. One of the main hopes is
that by generating a number of subgradients at various points in
some neighborhood of $x$, the behavior of $f$ around $x$ can roughly be
assessed. In the case of a convex function $f$ this is not just
wishful thinking, and a number of algorithms, especially those of
bundle type (e.g., Lemarechal 1975 and Wolfe 1975) rely on such an
approach. In the nonconvex case, however, there is the possibility,
without further assumptions on $f$ than local Lipschitz continuity,
that the multifunction $\del^Cf:  x\mapsto \del^Cf(x)$
may be rather bizarrely disassociated from $f$.
An example given at the end of this section has
$f$ locally Lipschitzian, yet such that there exist many other locally
Lipschitzian functions $g$, not merely differing from $f$ by an additive
constant, for which $\del^Cg(x) = \del^Cf(x)$ for all $x$.
Subgradients alone cannot discriminate between the properties of
these different functions
and therefore cannot be effective in determining their local minima.
\if{
[...] The key seems to lie in postulating the existence of the ordinary
directional derivatives and some sort of relationship between them and $\del^Cf$.
Mifflin (1977a and 1977b), most notably has worked in this direction.}\fi
\end{quote}
}

In this paper, we consider 
lower semicontinuous functions and arbitrary subdifferentials.
We identify a large subclass of lower semicontinuous functions
whose radial subderivative at a given point of their domain
can be fully expressed in terms of the
subdifferential of the function at neighbouring points.
We show that these functions are precisely the functions
whose radial subderivative satisfies a mild directional upper
semicontinuity property, independently from any subdifferential. Such functions are said to be
upper semismooth. This class includes the proper lower semicontinuous
(directionally, approximately) convex functions,
the regular directionally Lipschitz functions,
the Mifflin semismooth functions,
the Thibault-Zagrodny directionally stable functions.

We show that, as expected, the class of upper semismooth functions
satisfies the subdifferential determination property,
that is, if two functions have the same subdifferential
and one of the functions is upper semismooth, then the two functions
are equal up to an additive constant.
\if{A class of functions $\Fcal$ is said to have the
\textit{subdifferential determination property on an
open subset $\Omega \subset X$} if for any
two lsc functions $f, g:X\to \Rex$ with $g$ in the class $\Fcal$,
the following holds:
\begin{equation}\label{suddiffdet}
\del f(x)=\del g(x) \text{ for all } x\in \Omega
\Longrightarrow f=g + Const  \text{ on  }\Omega.
\end{equation}
}\fi
J.-J. Moreau \cite{Mor65} was the first to consider this property
for the class of proper lower semicontinuous convex functions defined on
Hilbert spaces. His result was later extended by Rockafellar \cite{Roc70}
to the same class of functions defined
on general Banach spaces. Since then
this property has been the object of intensive research and
various classes of non convex functions have been considered in this context.
We refer to the papers by L. Thibault and D. Zagrodny \cite{TZ05,TZ10}
for a detailed account of the history of this property.

The technique we use to prove the subdifferential
determination property is simple: from the subdifferential assumption on
the two functions, one of them being upper semismooth,
we derive an inequality between the radial subderivatives of
the functions; we then conclude by invoking a mean value theorem with
Dini subderivatives. The structure of the paper is as follows.
In Section \ref{Dinisect}, we revisit the mean value theorems with Dini
subderivatives. In Section \ref{subsubsect}, we discuss subderivatives
and subdifferentials and recall the duality formula linking them.
In Section \ref{subsubsect}, we define the class of upper semismooth functions
and give the main examples of functions in this class.
In Section \ref{determinationsect}, we prove our main theorems on the 
subdifferential determination property.
\section{Mean value theorems with Dini subderivatives}\label{Dinisect}
The lower right-hand Dini derivative,
or lower radial subderivative from the direction $d=+1$,
of a function $\psi:\R\to\xR$
at a point $t_0\in\R$ where $\psi$ is finite is denoted by
$$
\psi^r(t_0;+1):=\liminf_{t\searrow 0}\frac{\psi(t_0+t)-\psi(t_0)}{t}.
$$
Its upper version is denoted by
$$
\psi^r_+(t_0;+1):=\limsup_{t\searrow 0}\frac{\psi(t_0+t)-\psi(t_0)}{t}.
$$
\begin{proposition}[Mean value inequality]\label{mvi}
Let $\psi:[0,1]\to\xR$ be lower semicontinuous on $[0,1]$ and finite at $0$.
Then, for every real number $\lambda\le \psi(1)-\psi(0)$, there
exists $t_0\in [0,1{[}$ such that $\psi(t_0)\le \psi(0)+t_0\lambda$ and
$\lambda\le \psi^r(t_0;+1).$
\end{proposition}

\begin{proof}
For completeness, we recall the elementary argument,
as given, e.g., in \cite[Lemma 4.1]{JL13} or in \cite[Lemma 3.1]{JL14}.
For $t\in [0,1]$, let $g(t):=\psi(t)-t \lambda$.
The function $g:[0,1]\to \xR$ is lsc
on the compact $[0,1]$ and $g(0)=\psi(0)\le \psi(1)-\lambda=g(1)$.
Hence $g$ attains its
minimum on $[0,1]$ at a point $t_0\ne 1$.
Consequently, $\psi(t_0)-t_0\lambda=g(t_0)\le g(0)=\psi(0)$ and
since $g(t_0+t)\ge g(t_0)$ for every $t\in {]}0,1-t_0]$,
it follows that
\[
\forall t\in {]}0,1-t_0],\quad \frac{\psi(t_0+t)-\psi(t_0)}{t} \ge \lambda.
\]
Passing to the limit inferior as $t\searrow 0$, we get
$\psi^r(t_0;+1)\ge \lambda$ as claimed.
\end{proof}

\begin{proposition}[Mean value theorem: semicontinuous version]
\label{recov-subdiv-basic}
Let $\varphi:[0,1]\to\xR$ be lower semicontinuous on $[0,1]$ and finite at $0$
and let
$\gamma:[0,1]\to [-\infty,+\infty{[}$ be upper semicontinuous on $[0,1]$
and finite at $0$.
Assume that for every $t\in [0,1[$,
there exists a real number $\rho(t)$ such that
$\varphi^r(t;+1)\le \rho(t)\le \gamma^r(t;+1)$.
Then
\begin{equation*}
\varphi(1)-\varphi(0)\le \gamma(1)-\gamma(0).
\end{equation*}
\end{proposition}

\begin{proof}
Let $\psi:=\varphi-\gamma$. Then $\psi:[0,1]\to\xR$ is lower semicontinuous
on $[0,1]$ and finite at $0$.
So, according to Proposition \ref{mvi}, for every real number
$r\le \psi(1)-\psi(0)$, there
exists $t_0\in [0,1[$ such that $r\le \psi^r(t_0;+1)$.
There exists a real number $\rho(t_0)$ such that
\begin{equation}\label{recov-subdiv11}
\varphi^r(t_0;+1)\le \rho(t_0)\le \gamma^r(t_0;+1).
\end{equation}
Since
\begin{align*}
r\le \psi^r(t_0;+1)=
\liminf_{t\searrow 0}\left(\frac{\varphi(t_0+t)-\varphi(t_0)}{t}
                                 -\frac{\gamma(t_0+t)-\gamma(t_0)}{t}\right)
\end{align*}
and since, by (\ref{recov-subdiv11}),
\begin{align*}
\rho(t_0)\le  \liminf_{t\searrow 0}\frac{\gamma(t_0+t)-\gamma(t_0)}{t},
\end{align*}
for every $\eps>0$ one can find $t_\eps>0$ such that for all $t\in ]0,t_\eps]$
\begin{align*}
r-\eps< \frac{\varphi(t_0+t)-\varphi(t_0)}{t}
                                 -\frac{\gamma(t_0+t)-\gamma(t_0)}{t}
\quad\text{ and }\quad
\rho(t_0)-\eps\le \frac{\gamma(t_0+t)-\gamma(t_0)}{t},
\end{align*}
hence, for all $t\in {]}0,t_\eps]$,
\begin{align*}
r+\rho(t_0)-2\eps< \frac{\varphi(t_0+t)-\varphi(t_0)}{t}.
\end{align*}
Passing to the limit inferior, we get
\begin{align*}
r+\rho(t_0)-2\eps\le \varphi^r(t_0;+1).
\end{align*}
Since $\varphi^r(t_0;+1)\le \rho(t_0)$ by (\ref{recov-subdiv11}),
it follows that $r-2\eps\le 0$.
As $\eps>0$ was arbitrary, we conclude that $r\le 0$.
Therefore, for every real number $r\le \psi(1)-\psi(0)$ one has $r\le 0$,
proving that
$\psi(1)-\psi(0)\le 0$. This amounts to
$\varphi(1)-\varphi(0)\le \gamma(1)-\gamma(0)$, as claimed.
\end{proof}

\begin{remark}
(a) In Proposition \ref{recov-subdiv-basic}, the semicontinuity conditions
on $\varphi$ and $\gamma$ cannot be relaxed. Indeed, consider
the functions $\varphi,\gamma:[0,1]\to\R$
defined by
$$
\varphi(t):= 0 \quad\mbox{for all } t\in [0,1]
\quad \mbox{and}\quad 
\gamma(t):=\left\{
\begin{array}{ll}
1 & \mbox{if~~} t\in {[}0,1/2{[}\\
0 & \mbox{if~~} t\in {[}1/2,1{]}
\end{array}
\right..
$$
Then, $\varphi:[0,1]\to\R$ is continuous,
$\gamma:[0,1]\to\R$ is continuous except at point $1/2$ where it is
merely right-continuous, and
$\varphi^r(t;+1)= \gamma^r(t;+1)=0$ for all $t\in[0,1[$.
Yet the conclusion in Proposition \ref{recov-subdiv-basic} is false:
$\varphi(1)-\varphi(0)=0> \gamma(1)-\gamma(0)=-1$.

\smallbreak
(b) In Proposition \ref{recov-subdiv-basic}, the finiteness conditions
$\varphi^r(t;+1)<+\infty$ and $\gamma^r(t;+1)>-\infty$ for every $t\in [0,1[$, 
cannot be relaxed.
Indeed, consider the functions $\varphi,\gamma:[0,1]\to\R$
defined by
$$
\varphi(t):=
\left\{
\begin{array}{ll}
2 & \mbox{if~~} t\in {]}0,1]\\
0 & \mbox{if~~} t= 0
\end{array}
\right.
\quad \mbox{and}\quad \gamma(t)=\sqrt{t}\quad\mbox{for all } t\in [0,1].
$$
Then, $\varphi:[0,1]\to\xR$ is lower semicontinuous on $[0,1]$ with finite values,
$\gamma:[0,1]\to\R$ is continuous and
$\varphi^r(t;+1)\le \gamma^r(t;+1)$ for every $t\in [0,1[$.
Moreover, $\varphi^r(t;+1)$ and $\gamma^r(t;+1)$ are finite for every $t\in ]0,1[$,
but $\varphi^r(0;+1)= \gamma^r(0;+1)=+\infty$, and indeed
the conclusion of Proposition \ref{recov-subdiv-basic} is false: 
$\varphi(1)-\varphi(0)=2> \gamma(1)-\gamma(0)=1$.
\end{remark}

\begin{proposition}[Mean value theorem: continuous version]
\label{recov-subdiv-basic-continuous}
Let $\varphi,\gamma:[0,1]\to\R$ be continuous.
Assume there is a countable subset $C\subset [0,1]$ such
that for every $t\in [0,1]\setminus C$,
there exists a real number $\rho(t)$ such that
$\varphi^r_+(t;+1)\le \rho(t)\le \gamma^r(t;+1)$.
Then
\begin{equation*}
\varphi(1)-\varphi(0)\le \gamma(1)-\gamma(0).
\end{equation*}
\end{proposition}

\begin{proof}
We follow the pattern of the proof of \cite[(8.5.1)]{Die69}.
For any $\eps>0$, we will show that
$$
\varphi(1)-\varphi(0)\le \gamma(1)-\gamma(0) + 3\eps;
$$
the left hand side being independent of $\eps$, this will complete the proof.
Let $C:=\{c_n\tq n\in \N\}$ be the given countable subset of $[0,1]$.
Consider the set
$$
A:=\{ t\in [0,1]\tq \forall t'\in [0,t],\ 
\varphi(t')-\varphi(0)\le \gamma(t')-\gamma(0) + \eps t'+\eps \sum_{c_n<t'} 2^{-n}\}.
$$
It is clear that $0\in A$ and that if $t\in A$, then $[0,t]\subset A$.
Let $t_0=\sup A$. From the continuity of $\varphi$ and $\gamma$ it follows that
$t_0\in A$, so $[0,t_0]=A$. Therefore we need only prove that $t_0=1$.
\smallbreak
Suppose $t_0<1$.
If $t_0\not\in C$, there is a real number $\rho(t_0)$ such that
$$\varphi^r_+(t_0;+1)\le \rho(t_0)\le \gamma^r(t_0;+1),$$
so, by definition of the subderivatives, we can find $\eta>0$ such that
for every $s\in {]}0,\eta]$,
\begin{gather*}
\varphi(t_0+s)-\varphi(t_0)\le (\rho(t_0)+\eps/2)s\quad\text{and}\quad
(\rho(t_0)-\eps/2)s\le \gamma(t_0+s)-\gamma(t_0),
\end{gather*}
hence
$$
\varphi(t_0+s)-\varphi(t_0)\le \gamma(t_0+s)-\gamma(t_0)+\eps s,
$$
and since $t_0\in A$, we deduce
\begin{align*}
\varphi(t_0+s)-\varphi(0) &\le \gamma(t_0+s)-\gamma(0) +
\eps (t_0+s)+\eps \sum_{c_n<t_0} 2^{-n}\\
&\le \gamma(t_0+s)-\gamma(0) +
\eps (t_0+s)+\eps \sum_{c_n<t_0+s} 2^{-n},
\end{align*}
hence $t_0+\eta\in A$ contrary to the definition of $t_0$.
If $t_0\in C$, the set $\{n\in\N\tq t_0=c_n\}$ is not empty; by continuity
of $\varphi$ and $\gamma$, we can find $\eta>0$ such that
for every $s\in {]}0,\eta]$,
\begin{gather*}
\varphi(t_0+s)-\varphi(t_0)\le (\eps/2) \sum_{c_n=t_0} 2^{-n}
\quad\text{and}\quad
0\le \gamma(t_0+s)-\gamma(t_0)+(\eps/2) \sum_{c_n=t_0} 2^{-n},
\end{gather*}
hence from $t_0\in A$ we deduce again
\begin{align*}
\varphi(t_0+s)-\varphi(0) &\le \gamma(t_0+s)-\gamma(0) +
\eps t_0+\eps \sum_{c_n<t_0+s} 2^{-n}\\
&\le \gamma(t_0+s)-\gamma(0) +
\eps (t_0+s)+\eps \sum_{c_n<t_0+s} 2^{-n},
\end{align*}
which is a contradiction.
\end{proof}
\section{Subderivatives and subdifferentials}\label{subsubsect}

In the sequel, $X$ is a real Banach space,
$X^*$ is its topological dual,
and $\la .,. \ra$ is the duality pairing.
For $x, y \in X$, we let $[x,y]:=\{ x+t(y-x) \tq t\in[0,1]\}$;
the sets $]x,y[$ and $[x,y[$ are defined accordingly.
Set-valued operators $T:X\rightrightarrows X^*$
are identified with their graph $T\subset X\times X^*$.
All extended-real-valued functions $f : X\to\xR$ are assumed to be
lower semicontinuous (lsc)
and \textit{proper}, which means that
the set  $\dom f:=\{x\in X\tq f(x)<\infty\}$ is non-empty.

\medbreak
The framework, terminology and notation are the same as in our work
\cite{Las16}. For the reader's convenience, we briefly recall
the main definitions and facts.

For a lsc function $f:X\to\xR$, a point $\xb\in\dom f$ and
a direction $u\in X$,
we consider the following basic subderivatives:

- the (lower right Dini) \textit{radial subderivative}: 
\begin{equation*}
f^r(\xb;u):=\liminf_{t\searrow 0}\,\frac{f(\xb+tu)-f(\xb)}{t},
\end{equation*}
its upper version: 
\begin{equation*}
f^r_+(\xb;u):=\limsup_{t\searrow 0}\,\frac{f(\xb+tu)-f(\xb)}{t},
\end{equation*}
and its upper strict version (the \textit{Clarke subderivative}): 
\begin{equation*}
f^0(\bx;u):=  
\limsup_{t \searrow 0 \atop{(x,f(x)) \to (\bx,f(\bx))}}\frac{f(x+tu) -f(x)}{t};
\end{equation*}

- the (lower right Dini-Hadamard) \textit{directional subderivative}:
\begin{equation*}
f^d(\bx;u):=
\liminf_{t \searrow 0 \atop{u' \to u}}\frac{f(\bx+tu')-f(\bx)}{t},
\end{equation*}
and its upper strict version (the Clarke-Rockafellar subderivative): 
\begin{equation*}
f^\uparrow(\bx;u):= \sup_{\delta>0} 
\limsup_{t \searrow 0 \atop{(x,f(x)) \to (\bx,f(\bx))}}
\inf_{u' \in B_{\delta}(u)}\frac{f(x+tu') -f(x)}{t}.
\end{equation*}

It is immediate from these definitions that the following inequalities hold
($\rightarrow$ means $\le$):
\begin{align*}
f^r(\xb;u)  & \rightarrow f^r_+(\xb;u)\rightarrow f^0(\xb;u)\\
\uparrow \quad &     \qquad\qquad\qquad\quad\uparrow\\
f^d(\xb;u)  & \qquad\longrightarrow \quad\quad f^\uparrow(\xb;u)
\end{align*}
For $f$ locally Lipschitz at $\xb$,
one has $f^r(\xb;u)=f^d(\xb;u)$ and $f^0(\xb;u)=f^\uparrow(\xb;u)$.
For $f$ lsc convex, one has $f^d(\xb;u)=f^\uparrow(\xb;u)$.
A function $f$ satisfying such an equality is called \textit{regular}. 
However, in general, $f^d(\xb;u)<f^\uparrow(\xb;u)$.
\medbreak
Next, given a lsc function $f:X\to\xR$ and a point $\xb\in\dom f$,
we consider the following two basic subsets of the dual space $X^*$:

- the \textit{Moreau-Rockafellar subdifferential}
(the subdifferential of convex analysis):
\begin{equation*}\label{convex-sdiff}
 \del_{MR} f (\xb) :=
 \{ x^* \in X^* \tq \la x^*,y-\xb\ra + f(\xb) \leq f(y),\, \forall y \in X \};
\end{equation*}

- the \textit{Clarke subdifferential}, associated to the Clarke-Rockafellar subderivative:
\begin{eqnarray*}\label{Csub}
\partial_{C} f(\bx) := \{x^* \in X^* \tq \langle x^*,u\rangle \leq
f^\uparrow(\bx;u), \, \forall u \in X\}.
\end{eqnarray*}

All the classical subdifferentials (proximal,
Fr\'echet, Hadamard, Ioffe, Michel-Penot, \ldots)
lie between these two subsets,
and for a lsc convex $f$, all these subdifferentials coincide.
\ps
In the sequel, we call \textit{subdifferential} any operator $\del$ that associates 
a set-valued mapping $\partial f: X \rightrightarrows X^\ast$
to each function $f$ on $X$ so that
\begin{equation*}\label{inclusdansClarke}
\del_{MR} f\subset \partial f\subset \partial_{C} f
\end{equation*}
and the following \textit{Separation Principle}
is satisfied on $X$:
\medbreak
(SP)
\textit{For any lsc $f,\varphi$ with $\varphi$ convex Lipschitz
near $\xb\in\dom f $,
if $f+\varphi$ admits a local minimum at $\xb$, then
$0\in \delc f(\xb)+ \del \varphi(\xb),$ where
\begin{multline*}
\delc f(\xb):= \{\, \xb^*\in X^*\tq \mbox{there is a net }((x_\nu,x^*_\nu))_\nu\subset \del f \mbox{ with }\\
       (x_\nu,f(x_\nu))\to (\bx,f(\bx)),\ x^*_\nu\tow \bx^*,\ \limsup_\nu\,\la x^*_\nu,x_\nu-\xb\ra\le 0\,\}.
\end{multline*}
}
The Clarke subdifferential, the Michel-Penot (moderate) subdifferential and
the Ioffe subdifferential satisfy the Separation Principle in any Banach space.
The elementary subdifferentials (proximal, Fr\'echet, Hadamard, \ldots),
as well as their viscosity and limiting versions,
satisfy the Separation Principle in appropriate Banach spaces.
See, e.g. \cite{Iof12,JL13,Pen13} and the references therein.
\ps
Subdifferentials satisfying the Separation Principle
are densely defined:
\begin{theorem}[Density of subdifferentials]\label{dense}
Let $X$ be a Banach space,
$f:X\to\xR$ be lsc and $\xb\in\dom f$.
Then, there exists a sequence $((x_n,x_n^*))_n\subset\del f$ such that
$x_n\to \xb$, $f(x_n)\to f(\xb)$ and $\limsup_n \la x_n^*, x_n-\bx\ra\le 0$.
\end{theorem}
\begin{proof}
See \cite[Theorem 2.1]{JL14}.
\end{proof}
A sequence $(x_n)\subset X$ is said to be
\textit{directionally convergent to $\xb$ in the direction $v\in X$},
written $x_n\to_v \xb$,
if there are two sequences $t_n\searrow 0$ (that is, $t_n\to 0$ with $t_n>0$)
and $v_n\to v$ such that
$x_n=\xb + t_n v_n$ for all $n$.
Observe that for $v=0$, $x_n\to_v\xb$ simply means $x_n\to\xb$.
\ps
We call \textit{subderivative associated to a subdifferential $\del f$}
at a point $(\xb,u)\in \dom f\times X$, the \textit{support function} of
the set $\del f(\xb)$ in the direction $u$,
which we denote by
$$
f^\del(\xb;u):=
\sup \,\{\la \xb^*,u \ra \tq \xb^*\in\del f(\xb)\}.
$$

\ps
Subderivatives and subdifferentials are linked by the following formula
where, given $f:X\to\Rex$, we have denoted by $f':\dom f\times X\to \Rb$
any function lying between the subderivatives $f^d$ and $f^\uparrow$, that is: 
$f^d\le f'\le f^\uparrow$:

\begin{theorem}[Subderivative-subdifferential duality formula]\label{formula}
Let $X$ be a Banach space,
$f:X\to\xR$ be lsc, $\xb\in\dom f$ and $u\in X$.
Then, for any direction $v\in X$ and any real number $\alpha\ge 0$, one has
\begin{subequations}\label{formula0}
\begin{align}
\limsup_{x\to_v\xb} f^r(x;u+\alpha (\xb-x))&=
\limsup_{x\to_v\xb} f'(x;u+\alpha (\xb-x)) \label{formula0a}\\
&=\limsup_{x\to_v\xb}\,f^\del(x;u+\alpha (\xb-x)).
\label{formula0b}
\end{align}
\end{subequations}
\end{theorem}
\begin{proof}
See \cite[Theorem 3]{Las16}.
\end{proof}
\begin{remark}
(a) For $f$ locally Lipschitz at $\xb$, the formulas \eqref{formula0}
do not depend on $\alpha\ge 0$ since
\begin{equation*}\label{formula0lip}
\limsup_{x\to_v\xb} f^r(x;u+\alpha (\xb-x))=\limsup_{x\to_v\xb} f^r(x;u),
\end{equation*}
but they may depend on the direction $v\in X$:
for $f:x\in \R\mapsto f(x):=-|x|$ and $u\ne 0$, one has
$$\limsup_{x\to_u 0} f^r(x;u)=-|u|<\limsup_{x\to 0} f^r(x;u)=|u|.$$
\smallbreak

(b) For arbitrary lsc $f$, the value of the expressions
in \eqref{formula0}
may depend on $\alpha\ge 0$ even for convex $f$.
Indeed, in general for a lsc convex $f$ one has
$$f^r(\xb;u)<\limsup_{x\to\xb} f^r(x;u),$$
while always  (see \cite[Proposition 4]{Las16})
\begin{equation}\label{convex-uss}
f^r(\xb;u)=\inf_{\alpha\ge 0}\limsup_{x\to\xb} f^r(x;u+\alpha (\xb-x)).
\end{equation}
\end{remark}
\section{Upper semismooth functions}\label{usssect}

A lsc function $f:X\to\xR$ is said to be \textit{radially accessible} at a point $\xb\in\dom f$ from a direction $u\in X$ provided
$$f(\xb)=\liminf_{t\searrow 0}f(\xb+tu),$$
or equivalently, provided
there exists a sequence $t_n\searrow 0$ such that 
$f(\xb+t_n u)\to f(\xb)$.
Every convex lsc function $f:X\to\xR$
is radially accessible at $\xb$ from any $u\in X$ such that
$\xb+u\in \dom f$. On the other hand,
it is easily seen that if a lsc $f$ satisfies
$f^r(\xb;u)<+\infty$, then $f$ is radially accessible at $\xb$ from $u$.
The converse is not true: the function 
$f:t\in\R\mapsto \sqrt{|t|}$ is continuous, yet $f^r(0;u)=+\infty$.
For more examples and properties, see \cite{Las16}.

\begin{proposition}[Radial stability of the upper radial subderivative]\label{devdir}
Let $X$ be a Hausdorff locally convex space, $f:X\to\xR$ be lsc,
$\xb\in\dom f$ and $u\in X$
such that $f$ is radially accessible at $\xb$ from $u$.
Then, there is a sequence
$\mu_n\searrow 0$ such that $f(\xb+\mu_n u)\to f(\xb)$ and
\begin{equation}\label{applimvi1}
f^r_+(\xb;u)\le \liminf_{n\to +\infty}f^r(\xb+\mu_n u;u).
\end{equation}
In particular,
\begin{equation}\label{below00}
f^r_+(\xb;u)\le \inf_{\alpha\ge 0}
\limsup_{x\to_u\xb} f^r(x;u+\alpha (\xb-x)).
\end{equation}
\end{proposition}

\begin{proof}
See \cite[Proposition 7]{Las16}.
\end{proof}

Theorem \ref{formula}, Formula \eqref{convex-uss}
for convex lsc functions and Proposition \ref{devdir}
suggest to consider the following class of lsc functions.

\begin{defn}\label{defdsemismooth}
A lsc function $f:X\to\xR$ is said to be \textit{upper semismooth}
at a point $\xb\in\dom f$ in the direction $u\in X$
provided
\begin{align}
\inf_{\alpha\ge 0}\limsup_{x\to_u \xb} f^r(x; u +\alpha (\xb-x))\le f^r(\xb;u).
\label{dsemismooth}
\end{align}
\end{defn}

\begin{remark}\label{remdefdsemismooth}
(a) In case $f$ is locally Lipschitz at $\xb$, \eqref{dsemismooth}
boils down to
\begin{align}\label{dsemismoothLip}
\limsup_{x\to_u \xb} f^r(x; u)\le f^r(\xb;u). \tag{\ref{dsemismooth}Lip}
\end{align}
We then essentially recover the class of locally Lipschitz upper semismooth
functions considered by Borwein-Moors \cite[p.\ 305]{BM98a}
(with the slight difference that
the upper radial subderivative $f^r_+$ is used there
instead of the lower radial subderivative $f^r$ used here).
As observed in \cite[p.\ 305]{BM98a}, the terminology is
justified by the characterization of semismooth functions given by
Correa-Jofr\'e \cite[Corollary 6.3]{CJ89}
(see also Proposition \ref{exsemismooth}\,(d) below).
\ps
(b) In case $f$ is radially accessible at $\xb$ from the direction $u$,
the inequality \eqref{dsemismooth} becomes:
\begin{align}\label{dsemismoothequal}
f^r(\xb;u)=
\inf_{\alpha\ge 0}\limsup_{x\to_u \xb} f^r(x; u +\alpha (\xb-x)),
\tag{\ref{dsemismooth}bis}
\end{align}
and in addition, $f^r(\xb;u)=f^r_+(\xb;u)$,
i.e.\ the lower radial derivative and its upper version
coincide at $\xb$ in the direction $u$.
This follows by combining \eqref{dsemismooth} with the inequality
\eqref{below00} in Proposition \ref{devdir}.
\smallbreak
(c) One can have $f^r(\xb;u)=\pm\infty$ in \eqref{dsemismooth}.
For example, the continuous function
$x\mapsto \sqrt{|x|}$ is upper semismooth
at $\xb=0$ in the direction $u=1$ with $f^r(0;1)=+\infty$,
while the continuous function
$x\mapsto -\sqrt{|x|}$ is upper semismooth
at $\xb=0$ in the direction $u=1$ with $f^r(0;1)=-\infty$.
\end{remark}
\medbreak
We shall also consider a strict variant of the above notion:

\begin{defn}\label{defsemismooth}
A lsc function $f:X\to\xR$ is said to be \textit{strictly upper semismooth}
at a point $\xb\in\dom f$ in the direction $u\in X$
provided
\begin{align}
\inf_{\alpha\ge 0}\limsup_{x\to \xb} f^r(x; u +\alpha (\xb-x))\le f^r(\xb;u).
\label{semismooth}
\end{align}
In fact, equality holds in \eqref{semismooth}
since the reverse inequality is always true.
\end{defn}

\begin{remark}\label{remdefsemismooth}
(a) As above, in case $f$ is locally Lipschitz at $\xb$,
\eqref{semismooth} boils down to
\begin{align}\label{semismoothLip}
\limsup_{x\to \xb} f^r(x; u)\le f^r(\xb;u).
\tag{\ref{semismooth}Lip}
\end{align}
Since $\limsup_{x\to \xb} f^r(x; u)=f^0(\xb;u)$ according to
Borwein-Str{\'o}jwas \cite[Theorem 2.1]{BS89},
the inequality \eqref{semismoothLip} is actually equivalent to the equality
$f^r(\xb;u)=f^0(\xb;u)$.
The terminology is therefore justified since the latter equality
means that the lower radial subderivative and its upper strict version
coincide at $\xb$ in the direction $u$.
Locally Lipschitz functions satisfying such an equality in every direction
$u\in X$ are called (Clarke) \textit{regular at $\xb$} (see also below the extension
of this concept to the general case of lsc functions).
\smallbreak
(b) Evidently, \eqref{semismooth} is more demanding than \eqref{dsemismooth},
so every strictly upper semismooth function is upper semismooth.
The converse is not true:
the locally Lipschitz function $f:x\in\R\mapsto -|x|$ satisfies
$$\limsup_{x\to_u 0} f^r(x;u)=-|u|=f^r(0;u)<\limsup_{x\to 0} f^r(x;u)=|u|,$$
so $f$ is upper semismooth at $\xb=0$ from any $u\ne 0$ but not strictly
upper semismooth.
\end{remark}

We now proceed to give examples of strictly and non-strictly upper semismooth
functions.
Let us recall the definition of the concepts we shall consider.
Let $X$ be a Banach space,
$f:X\to\xR$ be lsc, $\xb\in\dom f$ and $u\in X$, $u\ne 0$.
Then $f$ is said to be:

\PT
\textit{semismooth} (Mifflin \cite{Mif77})
at $\xb$ provided $f$ is locally Lipschitz and for all $u\in X$,
$$((x_n,x^*_n))\subset \del_C f \mbox{ with }
x_n\to_u \xb \Rightarrow \la x^*_n,u\ra \to f^r(\xb;u).$$
\ps
\PT
\textit{directionally Lipschitz at $\xb$ with respect to $u$}
(Rockafellar \cite{Roc80}) if $f$ is lsc and
\begin{equation}\label{dirlip}
\limsup_{\stackrel{(x,f(x))\to (\xb,f(\xb))}{t\searrow 0, v\to u}} \frac{f(x+tv)-f(x)}{t}<\infty.
\end{equation}
\ps
\PT
\textit{regular at $\xb$} (Rockafellar \cite{Roc79})
if $f^d(\xb;v)=f^\uparrow(\xb;v)$ for every $v\in X$.
\ps
\PT
\textit{directionally approximately convex at $\xb$}
(see Daniilidis-Georgiev \cite{DG04},
Daniilidis-Jules-Lassonde \cite{DJL09} and the references therein), if for every
$u\in S_X$ and $\varepsilon>0$ there exists $\delta>0$
such that for all $x,y\in B(\xb,\delta)$, with $x\not=y$ and
$(x-y)/\|x-y\|\in B(u,\delta)$, and all $t\in [0,1]$
\begin{equation}\label{dac}
f(tx+(1-t)y)\leq tf(x)+(1-t)f(y)+\varepsilon t(1-t)\Vert x-y\Vert.
\end{equation}
(In finite-dimensional spaces, a
locally Lipschitz function is (directionally) approximately convex
if and only if it is lower-$C^1$,
cf.\ \cite{Spi81,DG04}.)

\begin{proposition}[Examples of (strictly) upper semismooth functions]
\label{exsemismooth}
Let $X$ be a Banach space,
$f:X\to\xR$ be lsc, $\xb\in\dom f$ and $u\in X$ with $u\ne 0$
and $\xb+u\in\dom f$.
Each of the following $f$ is strictly upper semismooth at $\xb$ in the direction $u$:
\ps
{\rm(a)}
$f$ is directionally Lipschitz at $\xb$ with respect to $u$ and regular at $\xb$;
\ps
{\rm(b)}
$f$ is convex;
\ps
{\rm(c)}
$f$ is directionally approximately convex.
\ps\noindent
The following $f$ is upper semismooth at $\xb$ in the direction $u$:
\ps
{\rm(d)}
$f$ is locally Lipschitz and semismooth at $\xb$.
\end{proposition}

\begin{proof}
(a)
For $f$ directionally Lipschitz at $\xb$ with respect to $u$, one has 
$f^\uparrow(\xb;u)=f^0(\xb;u)$
(cf.\ \cite[Theorem 3]{Roc80}). If moreover $f$ is regular at $\xb$,
that is $f^d(\xb;v)=f^\uparrow(\xb;v)$ for every $v\in X$,
we derive that $f^d(\xb;u)=f^0(\xb;u)$.
Then,
\begin{equation*}
f^r(\xb;u)\ge f^d(\xb;u)=f^0(\xb;u)=\limsup_{x\to \xb} f^r(x; u)\ge
\inf_{\alpha\ge 0}\limsup_{x\to \xb} f^r(x; u +\alpha (\xb-x)).
\end{equation*}
Hence \eqref{semismooth} is satisfied.
\ps
(b) We know from \cite[Proposition 4]{Las16} that,
if $f$ is \lsc\ convex, then
$$f^r(\xb;u)=\inf_{\alpha\ge 0}\limsup_{x\to\xb} f^r(x;u+\alpha (\xb-x)).$$
Hence, \eqref{semismooth} holds.
\ps
(c)
A directionally approximately convex function satisfies the following property
(cf.\ \cite[Proposition 1]{DJL09}):
for every $\epsilon>0$
and $u\in S_X$ there exists $\delta>0$ such that for all
$x\in B(\xb,\delta)$ and all $v\ne 0$ so that
$x+v\in B(\xb,\delta)$ and $v/\|v\|\in B(u,\delta)$, one has
\begin{equation}\label{technic2}
f^\uparrow(x;v)\le f(x+v)-f(x)+\eps\|v\|.
\end{equation}
Without loss of generality, we may assume that the given $u$ belongs to $S_X$.
Let $\epsilon>0$ and $\delta>0$ such that \eqref{technic2} holds.
Fix $0<t<\delta$ and consider any $x\in B(\xb,t\delta/2)$.
Let $v=tu+\xb-x$. Then $x+v\in B(\xb,\delta)$
and $v=t(u+(\xb-x)/t\in {]}0,\delta B(u,\delta){[}$.
It follows from \eqref{technic2} that for any such $x$,
\begin{equation*}
f^\uparrow(x;tu+\xb-x)\le f(\xb+tu)-f(x)+\eps\|tu+\xb-x\|.
\end{equation*}
Hence, for every $0<t<\delta$, since $f$ is \lsc\ at $\xb$,
\begin{equation*}
\limsup_{x\to\xb}f^\uparrow(x;tu+\xb-x)\le f(\xb+tu)-f(\xb)+t\eps.
\end{equation*}
So, for every $\eps>0$,
\begin{equation*}
\liminf_{t\searrow 0}
\limsup_{x\to\xb}f^\uparrow(x;\frac{tu+\xb-x}{t})\le
\liminf_{t\searrow 0}\frac{f(\xb+tu)-f(\xb)}{t}+\eps.
\end{equation*}
Consequently,
\begin{equation*}
\inf_{\alpha\ge 0}
\limsup_{x\to\xb}f^\uparrow(x;u+\alpha(\xb-x))\le
f^r(\xb;u).
\end{equation*}
Hence \eqref{semismooth} holds since we can replace $f^\uparrow$
by $f^r$ in the left hand side according to Theorem \ref{formula}.

\ps
(d)
By \cite[Corollary 6.3]{CJ89},
a locally Lipschitz function $f$ is semismooth if and only if
$$f^r(\xb;u)=\lim_{x\to_u \xb} f^r(x; u).$$
Hence (\ref{dsemismooth}Lip) holds.
\end{proof}

Besides the examples given in Proposition \ref{exsemismooth},
more elaborated classes of functions have been considered
to deal with the subdifferential determination property.
Classes of functions and results based on measure and integration theories
(e.g.\ Borwein-Moors \cite{BM98a} or Thibault-Zagrodny \cite{TZ10})
are discussed elsewhere.
Here, we discuss further the class of functions introduced by 
L. Thibault and D. Zagrodny in \cite{TZ05}:
given a subdifferential $\partial$, a lsc function
$g:X\to\xR$ is called \textit{$\del$-subdifferentially and directionally stable}
(\textit{sds} for short) on $\Omega$
provided that for every $u \in \Omega \cap\dom \del g$ and $v\in \Omega \cap\dom g$
with $v\ne u$, the following properties hold:
\begin{itemize}\itemsep-1ex\topsep0pt
\item[(i)] the function $t\mapsto \gamma (t):= g(u+t(v-u))$ is finite and
continuous on $[0, 1]$;
\item[(ii)] for any $t\in [0, 1{[}$, the right derivative
$\gamma_+'(t)$ exists and is less than $+\infty$;
\item[(iii)] for each fixed $y\in [u, v{[}$ and for each real number $\eps > 0$, there exists some
$r_0 \in {]}0, 1{[}$ such that for any $w=y+r(v-u)$ with
$r \in {]}0, r_0 ]$ and for every
$(x_n,x^*_n) \in \del g$ with $x_n \to x_0 \in [y, w{[}$ one has
\begin{equation}\label{sds}
\limsup_{n\to\infty}\, \la x^*_n,w-x_n\ra \le g^r(y;w-x_0)+\eps\|w-x_0\|.
\end{equation}
\end{itemize}

\begin{proposition}[sds implies strictly upper semi-smooth]\label{sds-uss}
Let $\Omega\subset X$ a nonempty open convex subset of a Banach space $X$
and let $f:X\to\Rex$ be a lsc function with $\Omega\cap \dom f\ne\emptyset$.
If the function $f:X\to\Rex$ is sds on $\Omega$, then
for every $\xb\in \Omega\cap \dom  \del f$, every $u\in X$ with
$\xb+u\in \Omega\cap \dom f$ and every $t\in [0,1{[}$,
the function $f$ is strictly upper semismooth at $\xb+tu$ in the direction $u$
and its radial subderivative $f^r(\xb+tu;u)$ is finite .
\end{proposition}

\begin{proof}
Let $f:X\to\Rex$ be sds on $\Omega$.
Let $\xb\in\Omega\cap\dom \del f$, $u\in X$ with $u\ne 0$,
such that $\xb+u\in \Omega\cap\dom f$.
We apply the above definition of a sds function with $u$ and $v$
respectively replaced by $\xb$ and $\xb+u$.
Let $\eps>0$ and let $\xb_t:=\xb+tu\in [\xb,\xb+u{[}$.
By condition (iii) of the definition, there exists
$r_0 \in {]}0, 1{[}$ such that for any $w=\xb_t+ru$ with
$r \in {]}0, r_0 ]$ and for every $(x_n,x^*_n) \in \del f$
with $x_n \to x_0 \in [\xb_t, w{[}$ one has
$$
\limsup_{n\to\infty}\, \la x^*_n,w-x_n\ra \le f^r(\xb_t;w-x_0)+
\eps\|w-x_0\|.
$$
Since subdifferentials are densely defined (Theorem \ref{dense}),
for every $x_0 \in [\xb_t, w{[}$ there does exist a sequence
$(x_n,x^*_n) \in \del f$ with $x_n \to x_0$. Given $x_0 \in [\xb_t, w{[}$,
write $w$ as $w=x_0+r_1 u$ with $r_1>0$ so that the above relation
can be written as
$$
\limsup_{n\to\infty}\, \la x^*_n,r_1 u+x_0-x_n\ra \le f^r(\xb_t;r_1u)+\eps\|r_1u\|.
$$
Dividing by $r_1$ we get that for every $x_0\in [\xb_t, w{[}$
and every $(x_n,x^*_n) \in \del f$ with $x_n \to x_0$,
\begin{equation*}
\limsup_{n\to\infty}\, \la x^*_n,u+\frac{x_0-x_n}{r_1       }\ra \le
f^r(\xb_t;u)+\eps\|u\|,
\end{equation*}
hence, for every $x_0\in [\xb_t, w{[}$,
\begin{equation*}
\inf_{\alpha\ge 0}\limsup_{x\to x_0} f^\del (x; u +\alpha (x_0-x)\ra\le f^r(\xb_t;u)+\eps\|u\|.
\end{equation*}
Invoking Theorem \ref{formula}, we conclude that for every $x_0\in [\xb_t, w{[}$, it holds
\begin{equation}\label{fin}
f^r(x_0;u)\le \inf_{\alpha\ge 0}\limsup_{x\to x_0} f^r (x; u +\alpha (x_0-x)\ra
\le f^r(\xb_t;u)+\eps\|u\|.
\end{equation}
The mean value inequality (Proposition \ref{mvi}) produces a point
$x_0\in [\xb_t, w{[}$ such that $f^r(x_0;u)>-\infty$, hence also
$f^r(\xb_t;u)>-\infty$ in view of \eqref{fin}.
Combining this with condition (ii) of the definition of sds, we see that
$f^r(\xb_t;u)$ is finite.
Next, considering \eqref{fin} with $x_0=\xb_t$ and noting that $\eps>0$
was arbitrary, we see that \eqref{semismooth} holds with $\xb=\xb_t$,
that is, $f$ is strictly upper semismooth at $\xb_t$ in the direction $u$.
\end{proof}

The next theorem states a key property of (strictly)
upper semismooth functions: roughly, a lsc function is (strictly)
upper semismooth at some point if and only if
its radial subderivative at this point can be recovered
from the values of the subdifferential
of the function at directional limiting points (at limiting points).

\begin{theorem}[Recovering the radial subderivative from a
subdifferential]\label{semismoothsubdiff}
Let $X$ be a Banach space,
$f:X\to\xR$ be lsc, $\xb\in\dom f$ and $u\in X$, $u\ne 0$.
Let also $\del$ be an arbitrary subdifferential.

{\rm (a)}
Assume $f$ is radially accessible at $\xb$ from $u$.
Then $f$ is upper semismooth at $\xb$ in the direction $u$ if and only if
\begin{align}\label{semismoothsubdiff1}
f^r(\xb;u)=\inf_{\alpha\ge 0}\limsup_{x\to_u \xb}
f^\del (x:u+\alpha (\xb-x)).
\end{align}

{\rm (b)} $f$ is strictly
upper semismooth at $\xb$ in the direction $u$ if and only if
\begin{align}\label{semismoothsubdiff2}
f^r(\xb;u)=\inf_{\alpha\ge 0}\limsup_{x\to \xb}
f^\del (x:u+\alpha (\xb-x)).
\end{align}
\end{theorem}

\begin{proof}
(a) By Theorem \ref{formula} with $v=u$,
the formulas (\ref{semismoothsubdiff1}) and (\ref{dsemismooth}bis) are the same.
\smallbreak
(b) We observed that equality holds in (\ref{semismooth});
this equality and \eqref{semismoothsubdiff2} are the same according to
Theorem \ref{formula} with $v=0$.
\end{proof}


\section{Subdifferential determination property}\label{determinationsect}

The two theorems of this section assert, with slightly different
assumptions, that the upper semismooth functions have the
subdifferential determination property.

\begin{theorem} [Subdifferential determination property]
\label{determination}
Let $X$ be a Banach space and $\Omega\subset X$ be a nonempty open convex subset.
Let $f,g:X\to\xR$ be lsc with $\Omega\cap \dom f\ne\emptyset$.
Assume that for every $\xb\in \Omega\cap \dom  \del f\cap \dom  \del g$ and
every $u\in X$, $u\ne 0$, with $\xb+u\in \Omega\cap \dom f\cap \dom g$,
the points $\xb_t:=\xb+tu$ satisfy the following properties:
\smallbreak
{\rm(\ref{determination}.1)}
$t\mapsto f(\xb_t)$ and $t\mapsto g(\xb_t)$ are continuous on $[0,1]$;
\smallbreak
{\rm(\ref{determination}.2)} there is a countable subset $C\subset [0,1]$ such
that for every $t\in [0,1]\setminus C$,
 either $f^r_+(\xb_t;u)$ or $g^r(\xb_t;u)$ is finite, and
$g$ is upper semismooth at $\xb_t$ in the direction $u$.
\smallbreak\noindent
Then,
\begin{equation}\label{impli}
\del f(x)\subset \del g(x) \mbox{ for all } x\in\Omega\cap \dom f
\Longrightarrow f=g + {\rm const}  \text{ on  }\Omega\cap \dom f.
\end{equation}
\end{theorem}

\begin{proof}
Assume
\begin{equation}\label{final0}
\del f(x)\subset \del g(x) \quad \mbox{for all } x\in\Omega\cap \dom \del f.
\end{equation}
The beginning of the proof is the same as in
Thibault-Zagrodny \cite[Theorem 3.21]{TZ10}.
We may suppose that $\Omega\cap \dom f$ is not a singleton because otherwise the result is obvious.
Then, if $\Omega\cap \dom f$ contains two distinct points $x,y$,
the set $\Omega\cap \dom \del f$ also contains two distinct points $\xb,\yb$ by the
density of $\dom\del f$ into $\dom f$ (Theorem \ref{dense}).
From \eqref{final0} it follows that $\xb$ and $\yb$ also belong to
$\dom \del g\subset \dom g$.
\ps
\textit{First step.} 
Fix $\xb\in\Omega\cap \dom \del f=\Omega\cap \dom  \del f\cap \dom  \del g$
and $\yb\in \Omega\cap \dom f\cap\dom g$, with $\xb\ne\yb$.
We claim that
\begin{equation}\label{fin1}
f(\yb)-f(\xb)\le g(\yb)-g(\xb).
\end{equation}
Let $u:=\yb-\xb$, hence $\xb_t=\xb+tu=\xb+t(\yb-\xb)$,
and let $t\in [0,1]\setminus C$.
By Assumption (\ref{determination}.1),
$f$ is radially accessible at $\xb_t$ from $u$, so in view of
Proposition \ref{devdir} and Theorem \ref{formula}
\begin{equation}\label{fin2}
f^r_+(\xb_t;u)\le \inf_{\alpha\ge 0}\limsup_{x\to_u \xb_t}\,
f^\del(x;u+\alpha(\xb_t-x)).
\end{equation}
By Assumption (\ref{determination}.2),
$g$ is upper semismooth at $\xb_t$ in the direction $u$,
hence according to Definition \ref{defdsemismooth} and Theorem \ref{formula}
\begin{equation}\label{fin3}
\inf_{\alpha\ge 0}\limsup_{x\to_u \xb_t}\,g^\del (x;u+\alpha(\xb_t-x))\le g^r(\xb_t;u).
\end{equation}
From \eqref{final0}, the right-hand side of \eqref{fin2} is less than
or equal to the left-hand side of \eqref{fin3}. We therefore conclude that
\begin{equation}\label{fin4}
f^r_+(\xb_t;u)\le g^r(\xb_t;u) \text{ for all } t\in [0,1]\setminus C.
\end{equation}

Now, consider the functions $\varphi:t\in [0,1]\mapsto f(\xb_t)$ and
$\gamma:t\in [0,1]\mapsto g(\xb_t)$. By Assumption (\ref{determination}.1),
they are finite and continuous on $[0,1]$
and, by Assumption (\ref{determination}.2), for every $t\in [0,1]\setminus C$,
either $\varphi^r_+(t;+1)$ or $\gamma^r(t;+1)$ is finite.
On the other hand, \eqref{fin4} can be reformulated as
\begin{equation}\label{fin4b}
\varphi^r_+(t;+1)\le \gamma^r(t;+1) \text{ for all } t\in [0,1]\setminus C.
\end{equation}
So we may invoke 
Proposition \ref{recov-subdiv-basic-continuous} to derive
that $\varphi(1)-\varphi(0)\le \gamma(1)-\gamma(0)$, that is,
\eqref{fin1} holds.
\ps
\textit{Second step.} 
In the first step, we have shown that \eqref{fin1}
holds for every point $\xb\in\Omega\cap \dom \del f$
and $\yb\in \Omega\cap \dom f\cap\dom g$.
Now, let $\xb\in\Omega\cap\dom f$ and $\yb\in \Omega\cap \dom f\cap\dom g$.
Applying the density Theorem \ref{dense}, we find
a sequence $(\xb_n)_n$ in $\dom \del f$ such that $\xb_n\to \xb$ and $f(\xb_n)\to f(\xb)$.
By \eqref{fin1}, for every $n\in \N$ and $\yb\in \Omega\cap \dom f\cap\dom g$ it holds
$$f(\yb)-f(\xb_n)\le g(\yb)-g(\xb_n).$$
Because $f(\xb_n)\to f(\xb)$ and $g$ is lower semicontinuous at $\xb$,
passing to the limit we get
$$f(\yb)-f(\xb)\le g(\yb)-g(\xb).$$
This inequality shows that $g(\xb)$ is finite whenever $\xb\in\Omega\cap\dom f$,
that is, $\Omega\cap\dom f= \Omega\cap\dom f\cap\dom g$.
So finally one has
$$
f(\yb)-f(\xb)\le g(\yb)-g(\xb) \text{ for all } \xb,\yb \in \Omega\cap\dom f.
$$
Then, interchanging the role of $\xb$ and $\yb$, we derive that in fact
$$
f(\yb)-f(\xb)= g(\yb)-g(\xb) \text{ for all } \xb,\yb \in \Omega\cap\dom f,
$$
which means that $f=g + {\rm const}$  on $\Omega\cap \dom f$.
The proof is complete.
\end{proof}

If instead of the continuous version of the mean value theorem
(Proposition \ref{recov-subdiv-basic-continuous})
we use the semicontinuous version
(Proposition \ref{recov-subdiv-basic}),
we obtain a variant of the above theorem with
weaker assumptions on $f$ but stronger on $g$. The proof being
similar will not be repeated.

\begin{theorem} [Subdifferential determination property: semicontinuous variant]
\label{determination2}
Let $X$ be a Banach space and $\Omega\subset X$ be a nonempty open convex subset.
Let $f,g:X\to\xR$ be lsc with $\Omega\cap \dom f\ne\emptyset$.
Assume that for every $\xb\in \Omega\cap \dom  \del f\cap \dom  \del g$ and
every $u\in X$, $u\ne 0$, with $\xb+u\in \Omega\cap \dom f\cap \dom g$,
the points $\xb_t:=\xb+tu$ satisfy the following properties:
\smallbreak
{\rm(\ref{determination2}.1)}
$t\mapsto g(\xb_t)$ is continuous on $[0,1]$;
\smallbreak
{\rm(\ref{determination2}.2)} 
for every $t\in [0,1{[}$, either $f^r(\xb_t;u)$ or $g^r(\xb_t;u)$ is finite, and
either {\rm(a)} or {\rm(b)} holds:
\\\hspace*{0.7cm}
{\rm(a)} $f$ is radially accessible at $\xb_t$ from $u$
and $g$ is upper semismooth at $\xb_t$ in direction $u$,
\\\hspace*{0.7cm}
{\rm(b)} $g$ is strictly upper semismooth at $\xb_t$ in direction $u$.
\smallbreak\noindent
Then,
\begin{equation*}
\del f(x)\subset \del g(x) \mbox{ for all } x\in\Omega\cap \dom f
\Longrightarrow f=g + {\rm const}  \text{ on  }\Omega\cap \dom f.
\end{equation*}
\end{theorem}

\begin{remark}
(a) Theorem \ref{determination} is new in the context of mean-valued based theorems.
Its assumption on $g$ is much weaker than the one in Theorem \ref{determination2}.
It should rather be compared and contrasted with integration-based results
such as those in \cite{BM98a,TZ10}. Since the technique and concepts are
totally different, this will be done in a separate paper.
\smallbreak
(b) Theorem \ref{determination2} unifies several results.
As an illustration, we mention three of them, which cannot
be derived from each other but which are all special cases
of Theorem \ref{determination2}, since the functions they involve
are either upper semismooth (case b1) or strictly upper semismooth
(cases b2 and b3).
\\\hspace*{0.7cm}
(b1) Correa-Jofr\'e \cite[Proposition 5.4]{CJ89}, where
$g$ is locally Lipschitz, semismooth and whose Clarke subdifferential
is single-valued at any point of a dense subset of $X$;
\\\hspace*{0.7cm}
(b2) Thibault-Zlateva \cite[Theorem 3.3]{TZl05}, where
lsc regular functions $g$ which are
continuous relative to their domains and strictly directionally Lipschitz
are shown to have a ''local subdifferential determination property'';
\\\hspace*{0.7cm}
(b3) Thibault-Zagrodny \cite[Theorem 4.1]{TZ05}
(in the special case $\gamma=0$), where $g$ is sds.
\end{remark}

{\small

}
\end{document}